\documentclass{amsart}
\usepackage{graphicx}
\usepackage[arrow,matrix,curve,color]{xy}
\usepackage{color}
\usepackage{tikz}
\definecolor{light-blue}{rgb}{0.8,0.85,1}
\definecolor{light-red}{rgb}{1,.4,.4}
\definecolor{purp}{rgb}{.7,.3,1}
\definecolor{yel}{rgb}{1,1,.5}
\definecolor{cy}{rgb}{0,1,1}
\newtheorem{theorem}{Theorem}
\newtheorem{corollary}[theorem]{Corollary}

\theoremstyle{definition}
\newtheorem{example}[theorem]{Example}

\newtheorem{remark}[theorem]{Remark}

\newcommand{\bR}{\mathbb R}

\newcommand{\bH}{\mathbb H}

\newcommand{\bZ}{\mathbb Z}

\newcommand{\bP}{\mathbb P}

\newcommand{\bQ}{\mathbb Q}

\newcommand{\tK}{\widetilde K}

\newcommand{\lp}{\textup{(}}
\newcommand{\rp}{\textup{)}}

\newcommand{\Sq}{\operatorname{Sq}}

\newcommand{\rank}{\operatorname{rank}}

\newcommand{\ev}{\text{\textup{even}}}
\newcommand{\co}{\colon\,}

\title[$K$-theory of $4$-complexes]{Complex $K$-theory of $4$-complexes}
\author{Jonathan Rosenberg}
\address{Department of Mathematics\\
University of Maryland\\
College Park, MD 20742-4015, USA} 
\email{jmr@umd.edu}
\urladdr{http://www2.math.umd.edu/\raisebox{-3pt}{~}jmr}
\makeatletter
\@namedef{subjclassname@2020}{%
  \textup{2020} Mathematics Subject Classification}
\makeatother

\begin{document}
\begin{abstract}
  This short note summarizes a number of facts about the ring $K^0(X)$ for
  $X$ a $4$-dimensional CW-complex.  Unusual features of this
  dimension are that every complex vector bundle is determined up to stable
  isomorphism by its Chern classes, that every even cohomology class 
  arises as a Chern class of a vector bundle, and that $K^0(X)$ is completely
  determined as a ring by knowledge of the even-dimensional cohomology
  ring $H^\ev(X; \bZ)$.  (All of these fail in high dimensions.)
\end{abstract}

\maketitle

\section{Introduction}
\label{sec:intro}
This short note was motivated by work the author has done 
\cite{MR4236456,Breakdown} in trying to study the famous
``gap-labeling conjecture'' of Bellissard \cite{bellissard:K}.
Study of a number of incomplete proofs of this conjecture
has led to the conclusion \cite[\S9]{MR4236456} that
if a tiling space $\Omega$ of dimension $d$ has
the property that the Chern
character $\text{ch}\co K^*(\Omega)\to H^*(\Omega; \bQ)$
factors through $H^*(\Omega; \bZ)$, then the conjecture holds.
This is always the case if $d\le 3$, so we were led to try
to find a counterexample to the conjecture in dimension $d=4$.
That is what led to a particular interest in the Chern character
for $4$-dimensional complexes.  (Tiling spaces are not usually
CW complexes, but they are inverse limits of finite complexes
via inverse systems that are easy to describe, so knowing
what happens for finite CW complexes would actually suffice for
purposes of the applications to tiling spaces.)

A special feature of dimension $4$ is that if $X$ is a
$4$-dimensional CW-complex, then the Atiyah-Hirzebruch spectral
sequence for computing $K^*(X)$ from $H^*(X; \bZ)$ collapses.
Nevertheless, $K^0(X)$ is \emph{not} necessarily isomorphic
to $H^\ev(X; \bZ)$, as one can see from the example of $X=\bR\bP^4$,
for which the reduced cohomology is all $2$-torsion whereas
$\tK(X)\cong \bZ/4$.  However, as we shall see, the structure
of $K^0(X)$ for $X$ a $4$-dimensional CW-complex can always be
computed, even as a ring, from knowledge of $H^\ev(X; \bZ)$ as a ring
under cup product.

We do not claim any great originality for the results here,
as a similar problem for $KO^0(X)$ was studied a long time ago
in \cite{MR123331}, and it's quite possible that some of what we do
here is already known to experts, but we have not been able to find explicit
references in the literature for our main theorems, so we thought it
would be good to document them. I would like to thank Michael Albanese
for some typo corrections and for pointing me to his post \cite{Alb}.

\section{Main Results}
\label{sec:results}

Throughout this paper, $X$ will be a connected $4$-dimensional CW
complex.  All of the interesting questions show up already when $X$ is
a finite complex.  First, we address the surjectivity of the Chern class
maps $c_1\co K^0(X)\to H^2(X;\bZ)$ and $c_2\co K^0(X)\to H^4(X;\bZ)$.
\begin{theorem}
\label{thm:Chernsurj}
Let $X$ be a connected $4$-dimensional CW complex.  Then for each
$x\in H^2(X;\bZ)$, there is a unique {\lp}up to isomorphism{\rp} line
bundle $L_x$ over $X$ with $c_1(L_x)=x$, and for each $y\in H^4(X;\bZ)$,
there is a unique {\lp}up to isomorphism{\rp} rank-$2$ complex vector 
bundle $V_y$ over $X$ with $c_1(V_y)=0$ and $c_2(V_y)=y$.  
\end{theorem}
\begin{proof}
  The first statement about line bundles works in any dimension, since
  $H^2$ always classifies complex line bundles.  The second statement,
  however, is specific to dimension $4$.  Note that the classifying space
  for rank-$2$ complex vector bundles with vanishing $c_1$ is
  $BSU(2)\cong \bH\bP^\infty$.  This has a CW structure with one cell
  in each dimension divisible by $4$, and the attaching maps given by
  the quaternionic Hopf fibrations.  Now $\bH\bP^\infty$ is $3$-connected
  and has $\pi_4(\bH\bP^\infty)\cong \bZ$.  Its next non-zero homotopy
  group is $\pi_5(\bH\bP^\infty)\cong \pi_4(SU(2))=\pi_4(S^3)\cong \bZ/2$.
  Starting with $\bH\bP^\infty$, we can construct a model for $K(\bZ,4)$
  by iteratively killing the higher homotopy groups, starting by killing
  $\pi_5$ by attaching a $6$-cell.  In this way we get an
  inclusion $\bH\bP^\infty\to K(\bZ, 4)$ where the larger space $K(\bZ,4)$
  has additional cells starting only in dimension $6$.  Thus for
  $X$ a $4$-dimensional CW complex, the map
  \[
    [X, BSU(2)]\xrightarrow{c_2} [X, K(\bZ,4)]\cong H^4(X;\bZ)
  \]
  is an isomorphism, and this gives the desired conclusion.
\end{proof}
\begin{remark} Note that in high dimensions, the Chern class maps
  are usually not surjective onto the integral cohomology.  For example,
  by \cite[Corollary 4.4]{VB}, the image of
  $c_n\co K^0(S^{2n})\to H^{2n}(X;\bZ)$ is a subgroup of index
  $(n-1)!$.
\end{remark}

The first part of following fact is well known,
but the part about the Chern classes may not be.
\begin{theorem}
\label{thm:Chernclasssplitting}
  Let $X$ be a connected $4$-dimensional CW complex.  Then the
  Atiyah-Hirzebruch spectral sequence $H^p(X;K^q)\Rightarrow K^{p+q}(X)$
  collapses at $E_2$ {\lp}i.e., there are no differentials{\rp}.
  The edge homomorphisms
  \[
  H^2(X;\bZ)\to K^0(X)\quad \text{and}\quad H^4(X;\bZ)\to K^0(X)
  \]
  give splittings for the Chern class maps $c_1\co K^0(X)\to H^2(X;\bZ)$
  and $c_2\co K^0(X)\to H^4(X;\bZ)$.
\end{theorem}
\begin{proof}
  It is known that the first differential in the Atiyah-Hirzebruch
  spectral sequence for complex $K$-theory is $\Sq^3$, which vanishes
  identically on $4$-complexes.  So that proves the first statement.

  Now the Atiyah-Hirzebruch spectral sequence
  comes from the skeletal filtration $X^{(j)}$
  of $X$, so $E_\infty^{4,0}=K^0(X,X^{(2)})\cong K^0(X/X^{(2)})$,
  and $X/X^{(2)}$ is a wedge of $3$-spheres with $4$-cells attached. The
  edge homomorphism $K^0(X,X^{(2)})\to K^0(X)$ is just the pull-back map
  on $K$-theory under the quotient map $q\co X\to X/X^{(2)}$.  
  Since $X/X^{(2)}$ has only one even (reduced) cohomology group, namely $H^4$,
  $K^0(X/X^{(2)})\cong H^4(X/X^{(2)};\bZ)\cong H^4(X;\bZ)$, and
  because of Theorem \ref{thm:Chernsurj}, we see that $q^*$ is a splitting
  for $c_2$.  Similarly, $E_\infty^{2,0}=K^0(X^{(3)},X^{(1)})\cong
  \tK^0(X^{(3)})\cong H^2(X^{(3)};\bZ)$, and the edge homomorphism gives
  a splitting for $c_1$ again via Theorem \ref{thm:Chernsurj}.
\end{proof}
\begin{corollary}
\label{cor:Chernclass}
  Let $X$ be a connected $4$-dimensional CW complex, and let $V$ be a
  complex vector bundle over $X$.  Then $V$ is determined up to
  stable isomorphism by its rank {\lp}in $H^0(X;\bZ)\cong\bZ${\rp}, its first
  Chern class $c_1$ {\lp}in $H^2(X;\bZ)${\rp}, and its second
  Chern class $c_2$ {\lp}in $H^4(X;\bZ)${\rp}.
\end{corollary}
\begin{proof}
  This is immediate from the theorem, since $\tK^0(X)$ is an extension of
  $H^4(X;\bZ)$ by $H^2(X;\bZ)$, and these correspond in turn to $c_2$ and $c_1$
  by the last part of Theorem \ref{thm:Chernclasssplitting}.
\end{proof}
\begin{remark}
  A stronger version of Corollary \ref{cor:Chernclass} appears in
  \cite{Alb}, where it is shown that one can remove the word
  ``stable,'' and it is also shown that one can also strengthen Theorem
  \ref{thm:Chernsurj} to get a unique bundle of any rank $\ge 2$ with
  specified $c_1$ and $c_2$.
\end{remark}
  
Now we can state our main result: an algorithm for computing $K^0(X)$
as a ring given the cup product structure of the ring $H^\ev(X; \bZ)$.
Once again, we emphasize that these rings are not necessarily isomorphic,
just that the former can be computed from the latter.

\begin{theorem}
\label{thm:KfromH}
Let $X$ be a connected finite $4$-dimensional CW complex.  Then we
can describe $K^0(X)$ as a ring as follows.  It is generated over
$\bZ$ by the classes of the vector bundles $L_x$, $x\in H^2(X;\bZ)$,
and $V_y$, $y\in H^4(X;\bZ)$, as in Theorem \ref{thm:Chernsurj}.
These are subject to the following relations:
\begin{enumerate}
\item $[L_0]=1$, $[V_0]=2;$
\item $[L_x]\cdot [L_{x'}] = [L_{x+x'}]$, for $x,x'\in H^2;$
\item $[L_x]+ [L_{-x}] = [V_{-x^2}]$, for $x\in H^2;$  
\item $[V_y]+[V_{y'}]=2+[V_{y+y'}]$, for $y,y'\in H^4;$
\item $[V_y]\cdot[V_{y'}] = 2 + [V_{2y+2y'}]$, for $y,y'\in H^4;$
\item $[L_x]\cdot [V_y] = [L_{2x}] + [V_{x^2+y}] - 1$, for
  $x\in H^2$, $y\in H^4$;
\item $[L_x]+ [L_{x'}] = [L_{x+x'}] + [V_{xx'}] - 1$, for
	 $x,x'\in H^2$.
\end{enumerate}
\end{theorem}
\begin{proof}
  The classes of the $L_x$ and $V_y$ generate $K^0$, since
  $c(L_x\oplus V_y) = c(L_x)c(V_y)=(1+x)(1+y)=1+x+y$, and thus
  $L_x\oplus V_y$ has Chern classes $c_1=x$ and $c_2=y$, which are as
  general as possible, since $x\in H^2$ and $y\in H^4$ are unconstrained.
  (The rank can be adjusted by repeatedly adding or subtracting $1$.)
  Applying Corollary \ref{cor:Chernclass}, we see that
  $\bZ$-linear combinations of $1$, the $[L_x]$, and the $[V_y]$ exhaust $K^0$.
  
  (1) is clear, since $L_0$ and $V_0$ are trivial bundles of ranks $1$
  and $2$, respectively.  (2) is the usual relation in the topological
  Picard group. (3) follows from the relation
  \[
  c(L_x \oplus L_{-x}) = (1+x)(1-x)=1-x^2,
  \]
  together with Corollary \ref{cor:Chernclass}, which says that since
  $V_{-x^2}$ has the same rank and Chern classes as $L_x \oplus   L_{-x}$,
  they must coincide up to stable isomorphism.  Similarly, (4)
  follows from the calculation of
  \[
  c(V_y \oplus V_{y'}) = (1+y)(1+y')=1+y+y',\quad
  \rank(V_y \oplus V_{y'}) = 2 + 2 = 4.
  \]
  The rest of the multiplicative structure is determined,
  following a trick of Hirzebruch \cite[Theorem 4.4.3]{MR1335917}, by
  introducing, for each $y\in H^4$, a formal variable $\xi$ with
  $-\xi^2=y$. Note that if $\xi$ were an actual class in $H^2$, we
  would have a corresponding line bundle $L_\xi$ with
  $[L_\xi]+ [L_{-\xi}] = [V_y]$.  So we pretend that this is the case,
  getting
  $[L_x]\cdot [V_y] =
  [L_x\otimes L_\xi] + [L_x\otimes L_{-\xi}] = [L_{x+\xi}] + [L_{x-\xi}]$.
  Computing the Chern classes, we get
  \[
  c([L_x]\cdot [V_y]) = c([L_{x+\xi}] + [L_{x-\xi}])
  = (1 + x + \xi)(1 + x - \xi) = (1+x)^2 - \xi^2 = 1+2x+x^2+y.
  \]
  On the other hand,
  \[
  c([L_{2x}]+ [V_{x^2+y}]) = (1 + 2x)(1 + x^2 + y) = 1+2x+x^2+y.
  \]
  But $\rank([L_x]\cdot [V_y]) = 2$ and $\rank([L_{2x}]+  [V_{x^2+y}])=3$,
  so relation (6) follows, again by Corollary \ref{cor:Chernclass}.
  We can obtain relation (5) by a similar strategy, pretending that
  $[L_\xi]+ [L_{-\xi}] = [V_y]$ and $[L_{\xi'}]+ [L_{-\xi'}] = [V_{y'}]$.
  Then
  \[
  \begin{aligned}
    &[V_y]\cdot [V_{y'}] = ([L_\xi]+ [L_{-\xi}])\cdot([L_{\xi'}]+
    [L_{-\xi'}])\\
    &= [L_\xi]\cdot [L_{\xi'}] + [L_\xi]\cdot [L_{-\xi'}]
    + [L_{-\xi}]\cdot [L_{\xi'}] + [L_{-\xi}]\cdot [L_{-\xi'}]\\
    &= [L_{\xi+\xi'}] + [L_{\xi-\xi'}] + [L_{-\xi+\xi'}] +
    [L_{-\xi-\xi'}].
  \end{aligned}
  \]
  Now compute the Chern classes.  We obtain
  \[
  \begin{aligned}
    &(1+{\xi+\xi'})(1+{\xi-\xi'})(1{-\xi+\xi'})(1{-\xi-\xi'})\\
    &= (1-(\xi+\xi')^2)(1-(\xi-\xi')^2)\\
    &= (1-(\xi^2+\xi'^2+2\xi\xi')(1-(\xi^2+\xi'^2-2\xi\xi'))\\
    &= 1 - 2\xi^2 - 2\xi'^2 = 1 + 2y + 2y'.
  \end{aligned}
  \]
  Thus agrees with the Chern classes of $V_{2y+2y'}$, so (5) follows
  after we take the rank into account. (7), which generalizes (3),
  follows by the same strategy as (3); we have
  \[
	  c(L_x\oplus L_{x'})=(1+x)(1+x')=1+x+x'+xx',
  \]
  while
    \[
	    c(L_{x+x'}\oplus V_{xx'})=(1+x+x')(1+xx')=1+x+x'+xx',
  \]
so (7) follows after adjusting for the rank.
\end{proof}
\begin{example}
  Let $X=\bR\bP^4$,  Then $H^2(X;\bZ)\cong H^4(X;\bZ)\cong \bZ/2$.
  If $x$ and $y$ are additive generators of $H^2$ and $H^4$, then
  $x^2=y$.  So $K^0(X)$ is generated over $\bZ$ by $[L_x]$ and
  $[V_y]$. By Theorem \ref{thm:KfromH}, (2), (3), and (4), we have relations
  $[L_x]^2=1$, $2[L_x]=[V_y]$, $2[V_y]=4$.  So $u=[L_x]-1$ is an element of
  $\tK^0(X)$ of additive order $4$.  The multiplicative structure is
  determined by $(u+1)^2=1$, or $u^2=-2u$.  This agrees with the
  usual calculation in \cite[Corollary IV.6.47]{MR488029}.
\end{example}  
\bibliographystyle{amsplain}
\bibliography{KThy4complexes}
\end{document}